\documentclass{article}
\usepackage{amssymb}
\usepackage{amsfonts}
\usepackage{amsmath}
\usepackage{txfonts}

\newtheorem{theorem}{Theorem}

\newtheorem{definition}[theorem]{Definition}

\newtheorem{lemma}[theorem]{Lemma}

\newtheorem{remark}[theorem]{Remark}

\newenvironment{proof}[1][Proof]{\noindent\textbf{#1.} }{\ \rule{0.5em}{0.5em}}
\numberwithin{equation}{section}
\numberwithin{theorem}{section}

\begin{document}

\title{Stochastic derivative and heat type PDEs}

\author{Constantin Udri\c{s}te\thanks{University \textit{Politehnica} of \ Bucharest, Faculty of Applied Sciences,
Deptartment of Mathematics-Informatics I, eMail: udriste@mathem.pub.ro}, Virgil Damian\thanks{University \textit{Politehnica} of \ Bucharest, Faculty of Applied Sciences,
Deptartment of Mathematics-Informatics I, eMail: vdamian@mathem.pub.ro} \,
and Ionel \c{T}evy\thanks{University \textit{Politehnica} of \ Bucharest, Faculty of Applied Sciences,
Deptartment of Mathematics-Informatics I, eMail: tevy@mathem.pub.ro}}

\date{}

\maketitle

\begin{abstract}
In this paper we address again the problem of the connection between multitime Brownian sheet and heat type PDEs.
The main results include: the volumetric character of the solutions of the forward (backward) diffusion-like PDEs; 
the forward mean value of a Brownian process as the solution of the forward heat PDE; 
the backward mean value of a Brownian process as the solution of the backward heat PDE;  
the multitime stochastic processes with volumetric dependence; the stochastic partial derivative of a stochastic process
with respect to a multitime Wiener process; Hermite polynomials stochastic processes; 
union of Tzitzeica hypersurfaces (constant level sets of multitime
stochastic processes with volumetric dependence). The original results permit to extend 
the complete integrability theory to multitime stochastic differential systems, using path independent
curvilinear integrals and volumetric dependence.
\end{abstract}

\section{Introduction}

The papers \cite{cawa:75}, \cite{do:89} have studied stochastic curvilinear
integrals along all sufficiently smooth curves in $\mathbb{R}_{+}^{2}$. The
most simple situation is those of increasing curves. 

Our papers \cite{ud:07}, \cite{udda:11}, \cite{UDMT} have extended this point of view to stochastic curvilinear
integrals in $\mathbb{R}_{+}^{m}$ and to completely integrable stochastic
differential systems. These research trends and the original results are based on \textit{It\^{o}-Udri\c{s}te stochastic calculus rules} 
\begin{equation*}
dW_{t}^{a}~dW_{t}^{b}=\delta^{ab}\,c_{\alpha }(t)dt^{\alpha },dW_{t}^{a}~dt^{\alpha }=dt^{\alpha}~dW_{t}^{a}=0,dt^{\alpha }\,dt^{\beta }=0,
\end{equation*}%
where $a,b=\overline{1,d};\alpha ,\beta =\overline{1,m}$, the parameter $%
t=(t^{1},...,t^{m})\in \mathbb{R}_{+}^{m}$ means the \textit{multitime}, $%
\delta ^{ab}$ is the \textit{Kronecker symbol}, $v = t^{1}\cdots t^{m}$ is the volume of the hyperparallelepiped $\Omega_{0t}\subset R^m_+$,
$c_{\alpha }(t)=\frac{\partial v}{\partial t^{\alpha }}$ and the tensorial
product $\delta ^{ab}\,c_{\alpha }(t)$ represents the \textit{correlation
coefficients}. These formulas can be extended to stochastic processes on Riemannian manifolds
or on manifolds with a fundamental tensor of type $(2,0)$ that determines the correlation coefficients.
A similar single-time theory can be found in the paper \cite{udda:11a}. Also some ideas have the roots in the basic book \cite{pr}
containing the single-time theory of Brownian motion, stochastic integral and differential stochastic equations.

The main objective of the paper is the connection between multitime Brownian sheet and heat type PDEs.
Section 2 shows that the solutions of the forward diffusion-like PDEs system have a volumetric character
and that the forward mean value of a Brownian process is the solution of the forward heat PDE.
Section 3 proves that the solutions of the backward diffusion-like PDEs system have a volumetric character
and that the backward mean value of a Brownian process is the solution of the backward heat PDE.
Section 4 describes the differentiable multitime stochastic processes as sums of series of Hermite polynomials.
Section 5 studies unions of Tzitzeica hypersurfaces and their connection to differentiable multitime stochastic
processes. Some related deterministic ideas can be found in the papers \cite{LC1}, \cite{udarci:10}, \cite{udarci:11}.

\section{Forward diffusion-like PDEs system}

Let us now consider a real-valued function 
$f\left( t,x\right), \,t = (t^1,...,t^m) \in \mathbb{R}_{+}^{m},\,x = (x^1,...,x^d) \in \mathbb{R}^d$, with $f(t,0) = 0$,
which has continuous partial derivatives of the first order with respect to $t^{\alpha }$, and of the second order in $x^a$.
If
$$W_{t} = (W^a_t),\,\,t=(t^{1},...,t^{m})\in \mathbb{R}_{+}^{m},\,a = \overline{1,d}$$
is a multitime vectorial Wiener process \cite{ud:07}, then the function $f$ defines a 
stochastic process $f\left( t,W_{t}\right) ,~t\in \mathbb{R}_{+}^{m}$.
Let us accept (see \cite{all}) that $f(t,W_t)$ is a solution of the {\it forward diffusion-like PDEs system}
\begin{equation}
\frac{\partial f}{\partial t^{\alpha }}\left(t,W_{t}\right) - \frac{1}{2}\,c_{\alpha }(t)\,\Delta_x f (t,W_t) =0,\text{ for }\alpha =\overline{1,m},  \label{1.4}
\end{equation}
where
$$\Delta_x f=\frac{\partial ^{2}f}{\partial x^{a}\partial x^{b}}\, \delta ^{ab}.$$

\begin{theorem} The solution $f$ of the forward diffusion-like PDEs system (1)
depends on the point $(t^{1},...,t^{m})$ only through the product of
components $v=t^{1}\cdots t^{m}$, i.e., it is a function of the volume $%
v=t^{1}\cdots t^{m}$ of $\Omega _{0t}\subset \mathbb{R}_{+}^{m}$.
\end{theorem}

\begin{proof} From (1) it follows 
\begin{equation*}
c_{\beta }(t)\,\frac{\partial f}{\partial t^{\alpha }}\left( t,W_{t}\right)
=c_{\alpha }(t)\,\frac{\partial f}{\partial t^{\beta }}\left( t,W_{t}\right)
,\,\alpha \not=\beta .
\end{equation*}%
The general solution of this PDEs system is $f(t)=\varphi (t^{1}\cdots
t^{m},W_{t}),\,t\in \mathbb{R}_{+}^{m}$.

If $v=t^{1}\cdots t^{m}$ is the volume of $\Omega _{0t}\subset \mathbb{R}%
_{+}^{m}$ and $g\left( v,x\right) \overset{def}{=}\varphi \left( t^{1}\cdots
t^{m},x\right) $ has continuous partial derivatives of the first order in $v$
and of the second order in $x$, then from (1) it follows that $g$
satisfies the {\it forward multitime heat PDE} 
\begin{equation}
\frac{\partial g}{\partial v} - \frac{1}{2}\Delta g =0.  \label{1.5}
\end{equation}
\end{proof}

\begin{lemma}[see \cite{all}] Let $$W^x_{t} = (W^a_t)^x,\,\,t=(t^{1},...,t^{m})\in \mathbb{R}_{+}^{m},\,x \in R^d,\,a = \overline{1,d}$$
be a family of Brownian sheets with $P\{W^x_{t} = x\}=1$ for $t \in \partial R^m_+$. If $f: R^d \to R$ is a continuous 
bounded function and $u(t,x) = E[f(W^x_t)]$ is the forward mean value, then $u(t,x)$ 
is a unique bounded solution of forward heat type PDEs
$$\frac{\partial u}{\partial t^\alpha} (t,x)= \frac{1}{2}\,c_\alpha(t) \Delta_x u (t,x),\,\, t \in \stackrel{\circ}{R^m_+};$$
$$u(t,x) = f(x), \,\,t \in \partial R^m_+.$$
\end{lemma}

\begin{proof}. Let $K(t;x,y)$ be the transition density of an $m$-parameter
$R^d$- valued Brownian sheet $W^x(t) = W^x(t^1, . . . , t^m)$, starting on $\partial R^m_+$ from $x \in  R^d$,
under $P$, and moving to $y \in R^d$ at multitime $t = (t^1, . . . , t^m) \in R^m_+$, i.e., 
$$K(t;x,y) = \frac{1}{dy}P[W^x(t) \in dy] =\frac{1}{(2\pi\,t^1\cdots t^m)^{d/2}}\exp\left(\frac{- ||x - y||^2}{2\,t^1\cdots t^m}\right).$$
This function is a solution of the forward heat type PDEs, i.e.,
$$\frac{\partial}{\partial t^\alpha} K(t;x,y) = \frac{1}{2}c_\alpha \,\Delta_x K(t;x,y).$$

Taking into account the definition of the mean value $u(t,x)$, the assumptions on the function $f$ 
and the differentiability of the improper integral function, we find
$$
\frac{\partial u}{\partial t^\alpha} = \frac{\partial}{\partial t^\alpha}\int_{R^d}f(y)\,K(t;x,y)\,dy = \int_{R^d} f(y) \frac{\partial}{\partial t^\alpha}K(t;x,y)\,dy
$$
$$= \frac{1}{2}c_\alpha \,\int_{R^d}f(y) \,\Delta_x K(t;x,y)\,dy$$
$$= \frac{1}{2}c_\alpha \,\Delta_x \int_{R^d}f(y) \,K(t;x,y)\,dy = \frac{1}{2}c_\alpha\, \Delta_x u.$$
By definition and by the continuity and boundedness of the function
$f$ we easily have $$\lim_{s\to t,\, y \to x}u(s, y) = u(t, x) = f(x),$$ for $t \in \partial R^m_+$
and $x \in R^d$. The existence and uniqueness Theorem for a Cauchy problem 
associated to heat PDE ends the proof.
\end{proof}

The mean value function u(t, x) is a function of the volume $v = t^1\cdots t^m$, i.e., $u(t,x) = g(v,x)$.

\section{Backward diffusion-like PDEs system}

Let us consider again a real-valued function $f\left( t,x\right) $ on $\mathbb{%
R}_{+}^{m}\times \mathbb{R}^d$, with $f(t,0)=0$, which has continuous partial
derivatives of the first order with respect to $t^{\alpha }$, $\alpha =%
\overline{1,m}$ and of the second order in $x^a$. If
$$W_{t} = (W^a_t),\,\,t=(t^{1},...,t^{m})\in \mathbb{R}_{+}^{m},\,a = \overline{1,d}$$
is a multitime vectorial Wiener process, then the function $f$ defines a stochastic process 
\begin{equation*}
y_{t}=f\left( t,W_{t}\right) ,~t\in \mathbb{R}_{+}^{m}.
\end{equation*}

By It\^{o}-Udri\c{s}te Lemma \cite{ud:07}, \cite{udda:11}, \cite{UDMT}, we can write 
\begin{equation}
dy_{t}=\left( \frac{1}{2}\,c_{\alpha }\left( t\right) \,\Delta_x f (t,W_t) + \frac{\partial f}{\partial
t^{\alpha }}\left( t,W_{t}\right) \right) dt^{\alpha }+\frac{\partial f}{%
\partial x^a}\left( t,W_{t}\right) dW^a_{t},  \label{1.3}
\end{equation}%
where
$$\Delta_x f=\frac{\partial ^{2}f}{\partial x^{a}\partial x^{b}}\, \delta ^{ab}.$$

In order that the stochastic process $y_{t}$ be a martingale, the drift coefficients in formula (1) must vanish,
i.e., $f$ is a solution of the {\it backward diffusion-like system} 
\begin{equation}
\frac{\partial f}{\partial t^{\alpha }}\left(t,W_{t}\right) + \frac{1}{2}\,c_{\alpha }\left( t\right)\,\Delta_x f (t,W_t) = 0,
\text{ for }\alpha =\overline{1,m}.  \label{1.4}
\end{equation}

\begin{theorem} The solution $f$ of the backward diffusion-like system (4)
depends on the point $(t^{1},...,t^{m})$ only through the product of
components $v=t^{1}\cdots t^{m}$, i.e., it is a function of the volume $%
v=t^{1}\cdots t^{m}$ of $\Omega _{0t}\subset \mathbb{R}_{+}^{m}$.
\end{theorem}

\begin{proof} From (4) it follows 
\begin{equation*}
c_{\beta }(t)\,\frac{\partial f}{\partial t^{\alpha }}\left( t,W_{t}\right)
=c_{\alpha }(t)\,\frac{\partial f}{\partial t^{\beta }}\left( t,W_{t}\right)
,\,\alpha \not=\beta .
\end{equation*}%
The general solution of this PDEs system is $f(t)=\varphi (t^{1}\cdots
t^{m},W_{t}),\,t\in \mathbb{R}_{+}^{m}$.
\end{proof}

Let us take the shape of volumetric features into account, recalling that a
volumetric function is invariant under the subgroup of central equi-affine
(i.e., volume-preserving with no translation) transformations, where the
determinant of the representing matrix is $1$.

If $v=t^{1}\cdots t^{m}$ is the volume of $\Omega _{0t}\subset \mathbb{R}%
_{+}^{m}$ and $g\left( v,x\right) \overset{def}{=}\varphi \left( t^{1}\cdots
t^{m},x\right) $ has continuous partial derivatives of the first order in $v$
and of the second order in $x$, then from (\ref{1.4}) it follows that $g$
satisfies the {\it backward multitime heat PDE} 
\begin{equation}
\frac{\partial g}{\partial v}+\frac{1}{2}\Delta g =0.  \label{1.5}
\end{equation}%
Consequently, the relation (\ref{1.3}) reduces to 
\begin{equation*}
dg\left( v,W_{t}\right) =\frac{\partial g}{\partial x^a}\left( v,W_{t}\right)
dW^a_{t}.
\end{equation*}%
Thus, if $\mathbb{E}\left[ || \frac{\partial g}{\partial x}\left(
v,W_{t}\right)||^{2}\right] $ is bounded with respect to $t$, in
bounded subsets of $\mathbb{R}_{+}^{m}$, then the stochastic process $%
g\left( v,W_{t}\right) $ is differentiable with respect to $W^a_t$ and the partial 
stochastic derivative is $\frac{\partial g}{\partial x^a}\left( v,W_{t}\right) $.

Let $T$ be the maturity multitime of the option.

\begin{lemma} Let $$W^x_{t} = (W^a_t)^x,\,\,t=(t^{1},...,t^{m})\in \mathbb{R}_{+}^{m},\,x \in R^d,\,a = \overline{1,d}$$
be a family of Brownian sheets with $P\{W^x_{t} = x\}=1$ for $t \in \partial R^m_+$. If $f: R^d \to R$ is a continuous 
bounded function and $w(t,x) = E_{t,x}[f(W^x_T)]$ is the backward mean value, then $w(t,x)$ is a 
unique bounded solution of backward heat type PDEs
$$\frac{\partial w}{\partial t^\alpha} (t,x)= - \frac{1}{2}\,c_\alpha(t) \Delta_x w (t,x),\,\, t \in \stackrel{\circ}{R^m_+};\,\,w(T,x) = f(x)$$
(terminal value problem).
\end{lemma}

\begin{proof} Let $V = T^1\cdots T^m$ and $v = t^1\cdots t^m$. The transition density is now
$$L(t; x, y) = \frac{1}{(2\pi\,(V - v))^{d/2}}\exp\left(\frac{- ||x - y||^2}{2\,(V - v)}\right).$$
\end{proof}

\begin{remark} 
The heat PDE has two variants, forward and backward, with the forward one being more common.
\end{remark}

\section{Path independent stochastic curvilinear integral}

The foregoing theory and the papers \cite{cawa:75}, \cite{ud:07}, \cite{udda:11}, \cite{UDMT}
suggest to introduce the notion of multitime differentiable stochastic processes.
For that we need a multitime vectorial Wiener process 
$$W_{t} = (W^a_t),\,\,t=(t^{1},...,t^{m})\in \mathbb{R}_{+}^{m},\,a = \overline{1,d}.$$

\begin{definition}
Let $\gamma _{0t}\subset \mathbb{R}_{+}^{m}$ be an increasing curve joining
the points $0, t \in R^m_+$. A multitime stochastic process $\Phi(t)
=\Phi(t,W_t),\,\,t\in \mathbb{R}_{+}^{m}$ is called differentiable with
respect to $W_t$, on $\mathbb{R}_{+}^{m}$, if there exists a multitime vectorial
adapted mesurable process $\phi(t) =(\phi_a (t, W_t)),\,\,t\in \mathbb{R}_{+}^{m}$
such that $\mathbb{E}\left[||\phi(t)||^{2}\right]$ is bounded for $t$ in
compact sets of $\mathbb{R}_{+}^{m}$ and 
\begin{equation}
\Phi (t)=\Phi (0)+\int_{\gamma _{0t}}\phi_a (s) \,dW^a_{s},  \label{holo_proc}
\end{equation}%
as a path independent stochastic curvilinear integral.
\end{definition}

In terms of stochastic differentials, the multitime stochastic process $%
\Phi_t$ is differentiable if the stochastic system 
\begin{equation*}
d\Phi (t)=\phi_a(t) \,dW^a_{t}.
\end{equation*}
is completely integrable, i.e., $\Phi (t)$ is a function of $v = t^1\cdots t^m$ and hence 
$\phi(t)$ is a vectorial function of $v$ \cite{cawa:75}, \cite{udda:11}.

The multitime process $\phi_a(t) $ is called the {\it stochastic partial derivative} of the multitime
process $\Phi (t) $ with respect to $W_t^a$ (see also \cite{cawa:75}, \cite{do:89}).

\begin{remark} The differentiable multitime stochastic processes have
properties similar to holomorphic functions: a differentiable process has a
differentiable derivative, the curvilinear integral primitive of a
differentiable process is differentiable, and each differentiable process
admits a power series expansion.
\end{remark}

\begin{theorem} {\it Let $w(t,x) = E_{t,x}[f(W^x_T)]$ be the backward mean value. 
For any Wiener process $Y_t$, the stochastic process $w(v, Y_t)$ is differentiable.}
\end{theorem}

\begin{proof} Based on the previous Lemma, the backward mean value function $w(v, Y_t)$ 
is a solution of the backward multitime heat PDE (terminal value problem)
\begin{equation}
\frac{\partial g}{\partial v} + \frac{1}{2}\Delta g =0,\,\, g(T,x) = f(x).  
\end{equation}
\end{proof}

\section{Hermite polynomials and stochastic processes}

There is a special class of solutions of the backward multitime heat
equation which will be particularly interesting in stochastic problems.
These are the Hermite polynomials (see also \cite{maud}, \cite{GZ:11}). 

To develope further our theory, we need a point $t=(t^{1},...,t^{m})\in \mathbb{R}_{+}^{m}$, $v=t^{1}\cdots t^{m}$, 
a point $x = (x^1,...,x^d)\in \mathbb{R}^{d}$, 
a point $n = (n_1,...,n_d)\in N^d_0$, called multiindex, and the notations $|n| = n_1 + ... + n_d$ and $n! = n_1!\cdots n_d!$.
Also for $\xi = (\xi^1,...,\xi^d)\in R^d$, we write $\xi^n = (\xi^1)^{n_1}...(\xi^d)^{n_d}$.
To simplify, the partial derivative
$$\frac{\partial^{|n|}}{\partial_{\xi^{n_1}}...\partial_{\xi^{n_d}}}f$$
will be denoted by $\partial_{\xi^n}f$. Now, let us use the generating function expansion 
\begin{equation}
e^{<x, \xi> -\frac{1}{2}v||\xi|| ^{2}}=\sum_{n\in N_0^d}\,H_{n}(v,x)\xi^{n},\,\,x,\xi \in R^d,\,\,v\in R_{+}.
\end{equation}
Since
$$H_n(v,x) = \frac{1}{n!}\,\partial_{\xi^n}\,e^{<x, \xi> -\frac{1}{2}v||\xi|| ^{2}}|_{\xi = 0},$$
it follows that $H_{n}(v,x)$ is a polynomial of degree $n$ in $x = (x^1,...,x^d)$; it is called the {\it $n$-th Hermite polynomial}.
On the other hand, the generating function $e^{<x, \xi> -\frac{1}{2}v||\xi|| ^{2}}$ is a tensor product of $d$ one-dimensional 
generating functions of the form $e^{x \xi - \frac{1}{2}v \xi^{2}}$. Consequently
$$H_n(v; x^1,...,x^d) = H_{n_1}(v; x^1)\cdots H_{n_d}(v; x^d),$$
where $H_{n_k}$ denotes the $n_k$-th one-dimensional Hermite polynomial.

Completing the square in the definition of generating function, 
the $n^{th}$ Hermite polynomial of the variables $v \in R_+$ and $x \in R^d$ can be written as 
\begin{equation}
H_{n}\left( v,x\right) =\frac{\left( - v\right) ^{n}}{n!}e^{\frac{||x||^{2}}{2v}}%
\partial_{\xi^{n}}e^{-\frac{1}{2v}||x - v \xi||^{2}}|_{\xi = 0} = \frac{\left( - v\right) ^{n}}{n!}e^{\frac{||x||^{2}}{2v}}%
\partial_{x^{n}}e^{-\frac{||x||^{2}}{2v}}.  \label{1.6}
\end{equation}
Here and in the rest of the paper, we write $v^n = v^{n_1}\cdots v^{n_d}$.

The sequence $\left\{ H_{n}(v,\cdot )\right\} _{n=0}^{\infty }$ is a
complete orthogonal system with respect to the weight $(2\pi v)^{\frac{-1}{2}%
}\,e^{\frac{- ||x||^{2}}{2v}}$. The orthogonality means 
\begin{equation}
\mathbb{E}\left[ H_{m}\left( v,W_{t}\right) H_{n}\left( v,W_{t}\right) %
\right] =\left\{ 
\begin{array}{c}
\displaystyle0,\,\,\,\text{ if }m\neq n, \\ 
\displaystyle\frac{v^n}{n!},\,\,\text{if }m=n.%
\end{array}%
\right.   \label{1.7}
\end{equation}%

We introduce the multiindex $k_a = (0,...,k,...,0)$, with $k$ in the place $a$.
Then, using the generating function expansion, taking the partial derivatives with respect to $v$ and $x^a$, and equating the
coefficients, we find 
\begin{equation*}
\frac{\partial }{\partial v}H_{n}= - \frac{1}{2}\sum_a H_{n-2_a},\,\,\frac{\partial }{\partial x^a}H_{n}=H_{n - 1_a}.
\end{equation*}

Consequently each Hermite polynomial $H_{n}(v,x)$ is a solution of the
backward multitime heat PDE. Thus

\begin{theorem} Each Hermite polynomial $H_n(v,W_t)$ is a
differentiable process and its partial derivative with respect to $W^a_t$ is $H_{n - 1_a}(v,W_t)$.
\end{theorem}

It follows that finite sums of Hermite polynomials processes are
differentiable processes. We extend this statement to series of Hermite
polynomials, following the ideas of Cairoli and Walsh \cite{cawa:75}.

\begin{theorem}
Suppose $\left\{ a_{n}\right\}$ is a multiindex family of real
numbers such that%
\begin{equation*}
\sum_{n\in N^d_0 }a_{n}^{2}\,\,\frac{v^n}{n!}<\infty,\,\, \text{ for all 
}v>0\,\text{(convergent series)} \text{.}
\end{equation*}%
Then, the process $\Phi_t$ defined by%
\begin{equation}
\Phi _{t}=\sum_{n\in N^d_0 } a_{n}H_{n}(v,W_t)  \label{1.10}
\end{equation}%
is differentiable relative to $W_t$ and its partial derivative $\phi_{ta}$ with respect to $W^a_t$ is given
by 
\begin{equation}
\phi _{ta}=\sum_{n - 1_a\in N^d_0} a_{n}H_{n - 1_a}( v,W_t).  \label{1.11}
\end{equation}%
The convergence of series is understood in $L^{2}$.
\end{theorem}

\begin{proof}
By the orthogonality relations (\ref{1.7}), we find 
\begin{equation*}
\mathbb{E}\left[ \left(\sum_{|n| = 0 }^{|p|}a_{n}H_{n}\left( v,W_t\right) \right)
^{2}\right] =\sum_{|n| = 0 }^{|p|}a_{n}^{2}\,\,\frac{v^n}{n!}\,.
\end{equation*}%
This mean value is bounded due to the convergence of the series in the right
hand member. It follows that the series (\ref{1.10}) converges in $L^{2}$
and the same is true for the series (\ref{1.11}). Consider now the sequence
of partial sums 
\begin{equation*}
\phi _{ta}^{(|p|)}=\sum_{|n - 1_a| = 0}^{|p|}a_{n}H_{n-1_a}\left( v,W_t\right),
\end{equation*}%
and let%
\begin{equation*}
\Phi _{t}^{(|p|)}=a_{0}+\int_{\gamma _{0t}^{\alpha }}\phi ^{(|p|)}_{sa}\,dW_s^a
=a_{0}+\sum_{|n -1_a| = 0}^{|p|}a_{n}H_{n}\left( v,W_t\right),
\end{equation*}%
where 
\begin{equation*}
\gamma _{0t}^{\alpha }:\, t^{1}=c^{1},...,t^{\alpha -1}=c^{\alpha -1},
t^{\alpha }=\tau ^{\alpha }\in \left[ 0,t^{\alpha }\right], t^{\alpha
+1}=c^{\alpha +1},...,t^{m}=c^{m}, 
\end{equation*}
is an increasing curve joining the points $0$ and $t$, with $c^\beta = const >0,\,\beta\not= \alpha$.
It follows 
\begin{equation*}
\underset{|p|\rightarrow \infty }{\lim }\Phi _{t}^{(|p|)}=\Phi _{t}\text{ in }%
L^{2}\text{.}
\end{equation*}%
To finish the proof, we need only check that 
\begin{equation*}
\underset{|p|\rightarrow \infty }{\lim }\int_{\gamma _{0t}^{\alpha }}\phi
^{(|p|)}_{sa}\,dW_s^a=\int_{\gamma _{0t}^{\alpha }}\underset{|p|\rightarrow \infty }{\lim 
}\phi ^{(|p|)}_{sa}\,dW_s^a\text{, }\alpha =\overline{1,m}.
\end{equation*}%
Again, by the orthogonality relation (\ref{1.7}), we have 
\begin{equation*}
\mathbb{E}\left[ \left( \phi _{ta}-\phi _{ta}^{(|p|)}\right) ^{2}\right]
=\sum_{|n|=|p|+1}^{\infty }a_{n}^{2}\,\,\frac{v^{n - 1_a}}{\left(n - 1_a\right) !}\,,
\end{equation*}%
and consequently 
\begin{equation*}
\mathbb{E}\left[ \left( \int_{\gamma _{0t}}\left( \phi_{sa} -\phi
^{(|p|)}_{sa}\right)\, dW^a_s\right) ^{2}\right] =
\end{equation*}%
\begin{equation*}
=\int_0^v\sum_{|n|=|p|+1}^{\infty }a_{n}^{2}\,\,\frac{
v^{n_1}\cdots v^{n_a -1} \cdots v^{n_d}}{\left( n -1_a\right) !}\,\,dv_a=\sum_{|n|=|p|+1}^{\infty }a_{n}^{2}\,\,\frac{v^n}{n!}.
\end{equation*}%
In other words, 
\begin{equation*}
\underset{|p|\rightarrow \infty }{\lim }\int_{\gamma _{0t}}\phi
^{(|p|)}_{sa}\,\,dW^a_s=\int_{\gamma _{0t}}\phi_{sa} \,\,dW^a_s.
\end{equation*}
\end{proof}

\begin{theorem}
Suppose that $\varphi \left( v,x\right) $ has continuous partial derivatives
of the first order in $v$ and of the second order in $x$, and that $\left\{
\varphi \left( v,W_{t}\right) \right\} $ is a differentiable process, where $%
v=t^{1}\cdots t^{m}$. Then, for each $t\in \mathbb{R}_{+}^{m}$, we can write 
\begin{equation*}
\varphi \left( v,W_{t}\right) =\sum_{|n|=0}^{\infty }a_{n}H_{n}\left(
v,W_{t}\right) ,
\end{equation*}%
where the convergence is understood in $L^{2}$ and, for $t=(t^{1},...,t^{m})%
\in \mathbb{R}_{+}^{m}$,%
\begin{equation}
a_{n}=\frac{n!}v^n\,\,\mathbb{E}\left[ \varphi \left( v,W_{t}\right)
H_{n}\left( v,W_{t}\right) \right] .  \label{1.13}
\end{equation}
\end{theorem}

\begin{proof}
Define the process $\Phi _{t}=\sum_{|n| = 0}^{\infty }a_{n}H_{n}\left(v,W_t\right)$. Due to the convergence 
\begin{equation*}
\sum_{|n| =0}^{\infty }a_{n}^{2}\,\,\frac{v^n}{n!}<\infty ,
\end{equation*}%
the series defining $\Phi_t$ converges in $L^{2}$ and the process $\Phi_t$ is
differentiable.

Since $H_{n}$ is an orthogonal sequence, we get 
\begin{equation*}
\mathbb{E}\left[ \Phi _{t}H_{n}\left( v,W_t\right) \right] =\mathbb{E}\left[
a_{n}H_{n}^{2}\left( v,W_t\right) \right] =a_{n}\,\,\frac{v^n}{n!}.
\end{equation*}%
But, by hypothesis,%
\begin{equation*}
\mathbb{E}\left[\varphi\left( v,W_t\right) H_{n}\left( v,W_t\right) \right] =%
\frac{v^n}{n!}\,\,a_{n}.
\end{equation*}%
Thus,%
\begin{equation*}
\mathbb{E}\left[ \left(\varphi\left( v,W_t\right) -\Phi _{t}\right)
H_{n}\left(v,W_t\right) \right] =0,
\end{equation*}%
for all $n\in \mathbb{N^{d}_{0}}$. Then, $\varphi\left( v,W_t\right) -\Phi_{t}\equiv 0$.
\end{proof}

\section{Union of Tzitzeica hypersurfaces}

A hypersurface $M\subset \mathbb{R}_{+}^{m},\,m\geq 3,$ is called \textit{%
Tzitzeica hypersurface} \cite{tz}, \cite{tzi}, provided there exists a constant $a\in \mathbb{R}$
such that we have $K=a\,d^{m+1}$, for all points $t=(t^{1},...,t^{m})\in M$,
where $K$ is the Gauss curvature of the hypersurface and $d$ is the distance
from the origin of the space to the tangent hyperplane to the hypersurface
at the current point $t$. Since the Gauss curvature $K$ describes the shape
of the hypersurface, a Tzitzeica hypersurface has a bending against the
tangent hyperplane in fixed proportion to the normal component of the
position vector $t$. The most simple Tzitzeica hypersurfaces are the
constant level sets $M_{c}:t^{1}\cdots t^{m}=c$ in $\mathbb{R}^{m}$ ($2^{m-1}$
connected components).

\begin{remark} The Gauss curvature of a Cartesian implicit surface 
\begin{equation*}
M_{c}:F(t^{1},t^{2},t^{3})=c
\end{equation*}%
in $\mathbb{R}^{3}$ is 
\begin{equation*}
K=\left[
[F_{3}(F_{3}F_{11}-2F_{1}F_{13})+F_{1}^{2}F_{33}][F_{3}(F_{3}F_{22}-2F_{2}F_{23})+F_{2}^{2}F_{33}\right. 
\end{equation*}%
\begin{equation*}
-\left. (F_{3}(-F_{1}F_{23}+F_{3}F_{12}-F_{y}F_{13})+F_{1}F_{2}F_{33})^{2}
\right] [F_{3}^{2}(F_{1}^{2}+F_{2}^{2}+F_{3}^{2})^{2}]^{-1}.
\end{equation*}%
The surface $M_{c}$ is curving like a paraboloid if $K(t)>0$, hyperboloid if 
$K(t)<0$, or a cylinder or plane if $K(t)=0$, near a point $%
t=(t^{1},t^{2},t^{3})\in M_{c}$.
\end{remark}

Let us show that the constant level sets of the functions $\varphi
(t^{1}\cdots t^{m},x)$, with respect to $t=(t^{1},...,t^{m})\in \mathbb{R}%
_{+}^{m}$, are Tzitzeica hypersurfaces or unions of simple Tzitzeica
hypersurfaces in $\mathbb{R}_{+}^{m}$ indexed by the points $x$.

\begin{theorem} \textit{If $\varphi :\mathbb{R}\rightarrow $}$\mathbb{R}$%
\textit{\ is a $C^{2}$ nonconstant function and $c$ is a noncritical value
of $\varphi $, then the constant level set $N_{c}:\varphi (t^{1}\cdots
t^{m})=c$ is a union of simple Tzitzeica hypersurfaces.}
\end{theorem}

\begin{proof}. Let $A_c$ the set of the solutions of the equation $\varphi
(v) = c$ and $v = t^1\cdots t^m$. Then the constant level set $N_c$ is the
union of the constant level sets $t^1\cdots t^m = k,\,\,k\in A_c$. If $c$ is
not a critical value of $\varphi$, then the set $N_c$ is a union of
hypersurfaces.
\end{proof}

\begin{remark} If $c$ is a critical value of $\varphi$, then the set $N_c$
is a union of constant level sets $t^1\cdots t^m = k,\,\,k\in A_c$, but
those level sets which correspond to $\varphi^{\prime}(k) = 0$ are not
hypersurfaces.
\end{remark}

\begin{theorem} Each section $x=c$ of the hypersurface $H_{n}(v,x)=0$ 
in $\mathbb{R}^{m+d}=\{(t,x)\}$ is a union of
Tzitzeica hypersurfaces in $\mathbb{R}_{+}^{m}$.
\end{theorem}

The constant level set 
\begin{equation*}
\mathbb{E}\left[ H_{n}\left( v,W_t\right) H_{n}\left( v,W_t\right)\right] =
c > 0
\end{equation*}
is a simple Tzitzeica hypersurface.


\end{document}